\documentclass[12pt]{amsart}
\usepackage{amsmath,amsfonts,amsthm}
\usepackage{amssymb,latexsym}
\usepackage{cite}
\usepackage{cases}
\usepackage[utf8]{inputenc}
\usepackage[numbers]{natbib} 
\usepackage{enumitem}
\usepackage{hyperref,bm,comment} 
\newtheorem{thm}{Theorem}[section]

\newtheorem*{re}{Remark}
\newtheorem{defn}[thm]{Definition}

\newcommand{\ord}{\mbox{ord}}

\newcommand{\C}{\mathbb{C}}
\newcommand{\Z}{\mathbb{Z}}
\newcommand{\N}{\mathbb{N}}

\newcommand{\im}{\operatorname{Im}}
\newcommand{\SL}{\operatorname{SL}}
\DeclareSymbolFont{largesymbol}{OMX}{yhex}{m}{n}
\DeclareMathAccent{\Widehat}{\mathord}{largesymbol}{"62}
\numberwithin{equation}{section}

\title[primitive representations for even rank quadratic forms]{Eisenstein series part of the primitive representations for even rank quadratic forms}
\author{Ben Kane}
\address{Department of Mathematics, University of Hong Kong, Pokfulam, Hong Kong}
\email{bkane@hku.hk}
\author{Luhao Xue}
\address{Department of Mathematics, University of Hong Kong, Pokfulam, Hong Kong}
\email{u3577120@connect.hku.hk}
\date{\today}
\begin{document}
\begin{abstract}
In this paper, we first investigate the relationship between the number of primitive representations of $n$ by quadratic forms and the number of non-primitive ones. Recall that the generating function for the number of representations is a modular form, which naturally splits into an Eisenstein series, giving the main asymptotic contribution, and a cusp form, contributing an error term. We hence obtain a theorem to deal with the Eisenstein series part with quadratic Dirichlet character when deriving the formula for the number of primitive representations of an integer $n$ by even rank quadratic forms from the number of non-primitive ones. Formulas for special cases are given as examples.
\end{abstract}

\thanks{The research of the first author was supported by grants from the Research Grants Council of the Hong Kong SAR, China (project numbers HKU 17303618 and 17314122).}
\maketitle

\section{Introduction}
Throughout this paper, we let $n \in \mathbb{N}$, $k \in \mathbb{N}$, $d \in \mathbb{N}$, $\bm{x}=(x_1,x_2,\cdots,x_k) \in \mathbb{Z}^k$, and $\bm{a}=(a_1,a_2,\cdots,a_k) \in \mathbb{N}^k$, where $\mathbb{N}$ is the set of natural numbers excluding 0. Meanwhile, we use $\mu$ to represent the \textit{m\"{o}bius function} and define the notation $\delta_{d|n}$ by letting
\begin{align*}
\delta_{d|n}=
\begin{cases}
1&\text{if $d\mid n$},\\
0&\text{if $d\nmid n$}.
\end{cases}
\end{align*}
We denote the number of \textbf{representations} of $n$ by the quadratic form $\sum_{i=1}^k{a_ix_i^2}$ by $r_{\bm a}(n)$. That is, $r_{\bm a}(n)$ is the number of solutions $\bm{x}\in\Z^k$ to the equation
\begin{equation}\label{eqn:nonpre}
\displaystyle\sum_{i=1}^{k}{a_ix_i^2}=n.
\end{equation}
In 1770, Lagrange proved that every natural number could be represented as the sum of four integer squares. In other words, the equation
\begin{equation}\label{eqn:foursquare}
x_1^2+x_2^2+x_3^2+x_4^2=n
\end{equation}
with $x_i\in\mathbb{Z}$ is solvable for every $n\in\mathbb{N}$ and hence $r_{(1,1,1,1)}(n)>0$. In 1834, Jacobi further extended this by explicitly determining the number of solutions to \eqref{eqn:foursquare}, yielding the formula
\begin{equation}\label{eqn:Jacobi}
r_{(1,1,1,1)}(n)=8\displaystyle\sum_{m\mid n,\,4\nmid m} m.
\end{equation}
Although every integer can be represented as a sum of four squares, not every integer is primitively represented, where a solution of the equation \eqref{eqn:nonpre} is called \textbf{primitive} if $\gcd(x_1,x_2,\cdots,x_n)=1$. In particular, if $8\mid n$, then no primitive representation exists, which can be verified by a simple argument involving congruences. In this paper, we are interested in formulas similar to \eqref{eqn:Jacobi} for the number $r_{\bm a}^p(n)$ of primitive solutions of the equation \eqref{eqn:nonpre}. In Section \ref{sec:example}, we derive an explicit formula 
\begin{equation}\label{eqn:prifour}
r^{p}_{\left(1,1,1,1\right)}(n)=8\cdot\left(1+\frac{1}{2}\delta_{2\mid n}-\delta_{4\mid n}-\frac{1}{2}\delta_{8\mid n}\right)\cdot n\cdot\displaystyle\prod_{2\ne p\mid n}\left(1+\frac{1}{p}\right)
\end{equation}
from Jacobi's result, where the product over $p$ runs over the odd prime factors of $n$. We will give a proof in Section \ref{sec:example}. Formulas resembling \eqref{eqn:prifour} may be obtained from the following result, which gives a natural inter-relationship between $r_{\bm{a}}^p(n)$ and $r_{\bm{a}}(n)$.
\begin{thm}\label{thm:relationship}
For any $k\in \mathbb{N}$, $d\in\mathbb{N}$, and $\bm{a} = (a_1,a_2,\cdots ,a_k) \in \mathbb{N}^k$, we have the following.
\begin{enumerate}[leftmargin=*,align=left,label={\rm(\arabic*)}]
\item The representation numbers $r_{\bm{a}}$ can be represented by a sum of primitive representation numbers $r^{p}_{\bm{a}}$ via
\begin{equation}\label{eqn:nonpri}
r_{\bm{a}}\left(n\right)=\displaystyle\sum_{d^2\mid n}{r^{p}_{\bm{a}}\left(\frac{n}{d^2}\right)},
\end{equation}
\item The primitive representation numbers $r_{\bm{a}}^p$ can be represented as a sum of the representation numbers $r_{\bm{a}}$ via 
\begin{equation}\label{eqn:prinon}
r_{\bm{a}}^{p}\left(n\right)=\displaystyle\sum_{d^2\mid n}{\mu(d)}{r_{\bm{a}}\left(\frac{n}{d^2}\right)}.
\end{equation}
\end{enumerate}
\end{thm}

\begin{re}
Equation \eqref{eqn:nonpri} and Equation \eqref{eqn:prinon} clearly show us the quantitative relationship between $r_{\bm{a}}(n)$ and $r_{\bm{a}}^{p}(n)$. We mainly focus on the equation \eqref{eqn:prinon} since we need to derive the formula for $r_{\bm{a}}^{p}(n)$ from $r_{\bm{a}}(n)$.
\end{re}
Formulas like \eqref{eqn:Jacobi} arise because the underlying quadratic form satisfies a certain local-to-global principle, wherein any form which is indistinguishable from it modulo any integer $N$ (i.e., locally equivalent) is actually indistinguishable from it over the integers (i.e., globally equivalent). In general, these is not the case for an arbitrary quadratic form, but formulas resembling the right-hand side of \eqref{eqn:Jacobi} are valid for a natural average of the number of representations of integers by quadratic forms which are locally indistinguishable from the given quadratic form modulo global equivalence. In the theory of modular forms, these averages are Fourier coefficients of Eisenstein series. In order to use Theorem \ref{thm:relationship} (2) to obtain explicit formulas like \eqref{eqn:prifour} for the corresponding weighted averages of primitive representations when $k$ is even, we compute the associated weighted averages of the Fourier coefficients of Eisenstein series in the next theorem.
\begin{thm}\label{thm:mainthm}
Let $\psi$ be a Dirichlet character modulo $N$, and let $\varphi$ be a Dirichlet character modulo $M$. Let $n$ be an arbitrary positive integer.\\
\begin{enumerate}[leftmargin=*,align=left,label={\rm(\arabic*)}]
\item Denote the prime factorization of $M$, $N$ and $n$ by
\begin{align*}
N&=\displaystyle\prod_{p_i}{p_i^{\alpha_i}}\displaystyle\prod_{u_s}{u_s^{a_s}},\quad
M=\displaystyle\prod_{q_j}{q_j^{\beta_j}}\displaystyle\prod_{u_s}{u_s^{b_s}},\\
n&=\displaystyle\prod_{r_k}{r_k^{\lambda_k}}=\displaystyle\prod_{p_i}{p_i^{\gamma_i}}\displaystyle\prod_{q_j}{q_j^{\nu_j}}\displaystyle\prod_{u_s}{u_s^{c_s}}\cdot n_1\\
&=\displaystyle\prod_{q_j}{q_j^{\nu_j}}\displaystyle\prod_{u_s}{u_s^{c_s}}\cdot n_2=\displaystyle\prod_{u_s}{u_s^{c_s}}\cdot n_3,
\end{align*}
where
\[{\gcd}(n_1,NM)={\gcd}(n_2,M)={\gcd}(n_3,N,M)=1,
\]
\[
\alpha_i,\beta_j,\lambda_k,a_s,b_s \in\mathbb{N} \quad and \quad h,\gamma_i , \nu_j,c_s \in \mathbb{N} \cup\{0\}.
\]
Notice that the $u_s$ are precisely the primes dividing $\gcd(M,N)$, the $p_i$ primes divide $N$ but not $M$, and the $q_j$ primes divide $M$ but not $N$. Assume that $d$ and $m$ are positive. Then we have
\begin{multline}\label{eqn:maineqn}
\displaystyle\sum_{d^2\mid n}{\mu (d)}\displaystyle\sum_{m\mid\frac{n}{d^2}}{\psi\left(\frac{n}{md^2}\right)\varphi(m)m^h}=n_2^h\varphi(n_2)\displaystyle\prod_{u_s}{\left(1-\delta_{u_s\mid n}-\delta_{u_s^2\mid n}+\delta_{u_s^3\mid n}\right)}\\
\times\displaystyle\prod_{p_i^2\mid n}{\left(1-\frac{\overline{\varphi(p_i)}^2}{p_i^{2h}}\right)}\cdot \displaystyle\prod_{q_j}{\left(\left(1-\delta_{q_j^2\mid n}\overline{\psi(q_j)}^2\right)\psi\left({q_j^{\nu_j}}\right)\right)}\\
\times \prod_{r_k\mid n_1}\left(\frac{\delta_{r_k^2\mid n}\left(1-\overline{\psi(r_k)}^2\right)}{\varphi(r_k^{\lambda_k})\overline{\psi(r_k^{\lambda_k})}r_k^{h\lambda_k}}\cdot \frac{1-\left(\varphi(r_k)\overline{\psi(r_k)} r_k^{h}\right)^{\lambda_k-1}}{1-\varphi(r_k)\overline{\psi(r_k)} r_k^{h}} + \left(1+\frac{\psi(r_k)\overline{\varphi(r_k)}}{r_k^h}\right)\right).
\end{multline}
\item  In particular, if $\psi$ is a real (quadratic) character, then we have
\[
\sum_{d^2\mid n} \mu(d)\sum_{m\mid \frac{n}{d^2}} \psi\left(\frac{n}{md^2}\right)\varphi(m)m^h=c(n)n_2^h\prod_{r_k}\left(1+\frac{\psi(r_k)\overline{\varphi(r_k)}}{r_k^h}\right),
\]
where 
\[
c(n):=\varphi(n_2)\prod_{u_s}\left(1-\delta_{u_s\mid n}-\delta_{u_s^2\mid n} + \delta_{u_s^3\mid n}\right)\prod_{q_{j}}\left(1-\delta_{q_j^2\mid n} \right)\psi\left(q_j^{\nu_j}\right)\prod_{p_i^2\mid n}\left(1-\frac{\overline{\varphi(p_i^2)}}{p_i^{2h}}\right).
\]

\end{enumerate}
\end{thm}
\begin{re}
As noted above Theorem \ref{thm:mainthm}, the sums $\sum_{m\mid n}{\psi\left(\frac{n}{m}\right)\varphi(m)m^h}$ are Fourier coefficients of certain Eisenstein series of integral weight (see Section \ref{sec:prelim}). The generating function for $r_{\bm{a}}(n)$ naturally decomposes into an Eisenstein series and a cusp form (see Section \ref{sec:decompose} and Theorems \ref{thm:eisen} and \ref{thm:eispace}), and these have integral weight when the rank $k$ is even.
\end{re}
The paper is organized as follows. In Section \ref{sec:prelim}, we give some preliminaries and necessary information about modular forms. In Section \ref{sec:relationshipthm}, we give a proof for Theorem \ref{thm:relationship}. In Section \ref{sec:MainThm}, we give a proof for Theorem \ref{thm:mainthm}. In Section \ref{sec:example}, we give some examples to show how one can combine Theorem \ref{thm:relationship} and Theorem \ref{thm:mainthm} to derive the formula for $r_{\bm a}^p(n)$ from $r_{\bm a}(n)$ whenever local and global equivalence are the same; we note that, more generally, one can use this method to obtain a formula for the contribution of the Eisenstein series part to the primitive representation numbers. In Appendix \ref{sec:app}, we give some additional examples to show the application of Theorem \ref{thm:relationship} and Theorem \ref{thm:mainthm} by using the formulas for $r_{\bm{a}}(n)$ provided by Liu Yiming.

\section{Preliminaries}\label{sec:prelim}
\subsection{Modular forms}
In this subsection, we mainly introduce the basic definition of modular forms and the generating functions for quadratic forms. Let $  \gamma=
\begin{pmatrix}
	a & b \\
	c & d
\end{pmatrix}\in \SL_{2}(\mathbb{Z})$ and $ \tau\in \mathcal{H} $, where $\mathcal{H}$ is the \textit{upper half plane}
\[
\mathcal{H} = \{\tau \in \mathbb{C}: \im(\tau) > 0\}.
\]
A \begin{it}fractional linear transformation\end{it} is defined by 
$\gamma (\tau):=\dfrac{a\tau+b}{c\tau+d}$. Let $N$ be a positive integer. The \textit{principal congruence subgroup of level N} is 
\[
\mathit\Gamma(N):=\left\{\begin{pmatrix}
	a & b \\
	c & d
\end{pmatrix}\in \SL_{2}(\mathbb{Z}): 
\begin{pmatrix}
	a & b \\
	c & d
\end{pmatrix}
\equiv
\begin{pmatrix}
	1 & 0 \\
	0 & 1
\end{pmatrix}
({\rm mod}\,N)
\right\}.
\]
\begin{defn}\label{thm:con}
A subgroup $\mathit\Gamma$ of $\SL_{2}(\mathbb{Z})$ is a \textbf{congruence subgroup} if $\mathit\Gamma(N) \subset \mathit\Gamma$ for some $N \in \mathbb{Z}^+$, in which case $\mathit\Gamma$ is a congruence subgroup of \textbf{level N}.
\end{defn}
We require the level $N$ congruence subgroups
\[
\mathit\Gamma_0(N):=\left\{\begin{pmatrix}
	a & b \\
	c & d
\end{pmatrix}\in \SL_{2}(\mathbb{Z}): 
\begin{pmatrix}
	a & b \\
	c & d
\end{pmatrix}
\equiv
\begin{pmatrix}
	* & * \\
	0 & *
\end{pmatrix}
({\rm mod}\,N)
\right\}
\]
and
\[
\mathit\Gamma_1(N):=\left\{\begin{pmatrix}
	a & b \\
	c & d
\end{pmatrix}\in \SL_{2}(\mathbb{Z}): 
\begin{pmatrix}
	a & b \\
	c & d
\end{pmatrix}
\equiv
\begin{pmatrix}
	1 & * \\
	0 & 1
\end{pmatrix}
({\rm mod}\,N)
\right\}.
\]
For any matrix $\gamma = \begin{pmatrix}
	a & b \\
	c & d
\end{pmatrix}\in \SL_{2}(\mathbb{Z})$, $\tau \in \SL_{2}(\mathbb{Z})$ and any integer $k$ and define the \textbf{weight-$k$ slash-operator} $[\gamma]_k$ on the function $f:\mathcal{H} \rightarrow \textbf{C}$ by
\[
f[\gamma]_k(\tau)=(c\tau +d)^{-k}f(\gamma (\tau)).
\]

\begin{defn}\label{thm:modular}
For any congruence subgroup $\mathit{\Gamma}\subseteq \SL_{2}(\mathbb{Z})$, a \textbf{modular form of weight} k, \textbf{with respect to} $\mathit{\Gamma}$  is a holomorphic function $ f:\mathcal{H}\to \mathbb{C}$ satisfying the following:
\begin{enumerate}[leftmargin=*,align=left,label={\rm(\arabic*)}]
\item
 For any $\gamma \in \mathit{\Gamma}$, $ f(\gamma(\tau))=(c\tau+d)^kf(\tau)$;
\item
$f[\alpha]_k$ is holomorphic at $\infty$ for all $\alpha \in \SL_{2}(\mathbb{Z})$.
\end{enumerate}
Since any congruence subgroup $\mathit\Gamma$ must contain some power of $T:=\left(\begin{smallmatrix}1&1\\ 0&1\end{smallmatrix}\right)\in\mathit\Gamma$ and conjugate groups $\alpha^{-1}\mathit\Gamma\alpha$ are also congruence subgroups, if $f$ is a modular form of weight $k$ for some congruence subgroup $\mathit\Gamma$, then for every $\alpha\in\SL_2(\Z)$ we have a \textbf{Fourier expansion}
\[
f[\alpha]_k(\tau)=\sum_{n=0}^{\infty} a_n e^{\frac{2\pi i n\tau}{M}}.
\] 
Here we omit the dependence on $f$ and $\alpha$ in the notation. We call $f$ a \textbf{cusp form} with respect to $\mathit\Gamma$ if $a_0=0$ in the Fourier expansion of $f[\alpha]_k$ for all $\alpha \in \SL_{2}(\mathbb{Z})$. The modular forms of weight k with respect to $\mathit\Gamma$ are form a vector space $M_k(\mathit{\Gamma})$ over $\C$ and we denote the subspace of cusp forms by $\mathcal{S}_k(\mathit\Gamma)$.
\end{defn}
\begin{defn}\label{thm:chispace}
For each Dirichlet character $\chi$ modulo $N$ define the $\chi\mbox{-}eigenspace$ of $M_k(\mathit{\Gamma}_1(N))$,
\[
M_k(N,\chi)=\{f \in M_k(\mathit{\Gamma}_1(N)):f[\gamma]_k=\chi(d_{\gamma})f \quad {\rm for\,\, all}\,\, \gamma \in \mathit{\Gamma}_0(N) \}.
\]
\end{defn}
We then have \cite[Exercise 4.3.4(a)]{first}
\[
M_k(\mathit{\Gamma}_1(N))=\bigoplus_{\chi}M_k(N,\chi),
\]
and the same result holds for cusp forms \cite[Exercise 4.3.4(b)]{first}
\[
\mathcal{S}_k(\mathit{\Gamma}_1(N))=\bigoplus_{\chi}\mathcal{S}_k(N,\chi).
\]
It also holds for Eisenstein series \cite[Exercise 4.3.4(c)]{first}
\[
\mathcal{E}_k(\mathit{\Gamma}_1(N))=\bigoplus_{\chi}\mathcal{E}_k(N,\chi).
\]
The space of \textbf{Eisenstein series} are defined as the orthogonal complement of the subspace of cusp forms inside a space of modular forms under an inner product known as the \textbf{Petersson inner product}, but we will simply investigate the Eisenstein series in this paper via a well-known basis (see Section \ref{sec:EisBasis} below).
\subsection{Quadratic forms and theta functions}
\begin{defn}\label{thm:qr}
We define a \textbf{quadratic form} $Q(\vec x)$ over a ring $R$ to be a degree 2 homogeneous polynomial
\[
Q(\vec x):=\displaystyle\sum_{1\leqslant i \leqslant j \leqslant n}c_{ij}x_ix_j
\]
in $n$ variables with coefficients $c_{ij}$ in $R$. Meanwhile we define the \textbf{representation numbers  $r_Q(m)$ (over $R$)} by
\[
r_Q(m) : = \#\{\vec x \in R^n : Q(\vec x) = m \}.
\]
We take $R=\Z$ throughout this paper and omit it in the notation.
\end{defn}
\begin{defn}\label{thm:theta}
Define the \textbf{theta series of} $Q$ as the Fourier series generating function for the representation numbers $r_Q(m)$ given by
\[
\Theta_Q(\tau):=\displaystyle\sum_{m=0}^{\infty}r_Q(m)e^{2\pi i m \tau}.
\]
\end{defn}
These theta series are modular forms (see \cite[Corollary 2.3.2]{Quadratic} for more details).
\begin{thm}\label{thm:thetamod}
For any $Q \in \mathbb{Z}^n$, $\Theta_Q(\tau)$ belongs to the vector space $M_{\frac{n}{2}}(N,\chi)$ for some Dirichlet character $\chi$ and $N \in \mathbb{N}$ when n is even.
\end{thm}

\subsection{Decomposition into Eisenstein series and cusp forms.}\label{sec:decompose}
By the theory of modular forms, we can write the theta series as
\begin{equation}\label{eqn:seriesd}
\Theta_Q(\tau)=E(\tau)+C(\tau)
\end{equation}
where $E(\tau)$ is an Eisenstein series and $C(\tau)$ is a cusp form. Comparing the $m^{\rm th}$ Fourier coefficient of the equation \eqref{eqn:seriesd}, we obtain the decomposition \cite[2.4 Asymptotic Statements about $r_Q(m)$, p.\hskip 0.1cm 31]{Quadratic}
\[
r_Q(m)=a_E(m)+a_C(m).
\]
Generally speaking, the Fourier coefficients of the Eisenstein series grow faster than the Fourier coefficients of the cusp form, so $a_E(m)$ gives a good approximation of $r_Q(m)$. As such, it is beneficial to obtain explicit formulas for a bases of spaces of Eisenstein series and then to determine the Eisenstein series $E$ occurring in \eqref{eqn:seriesd} as a linear combination of these basis elements, ultimately giving a formula for $a_E(m)$. 
\subsection{Bases of Eisenstein series.}\label{sec:EisBasis}
We require a basis of the weight $k$ Eisenstein series and their Fourier expansions.  Let 
\[
g(\chi):=\sum_{n\pmod{N}} \chi(n)e^{\frac{2\pi i n}{N}}
\]
 denote the  \textbf{Gauss sum} for a character $\chi$ of modulus $N$. Letting $\psi$ and $\varphi$ denote characters of moduli $u$ and $v$ such that $uv=N$, for $k\geq 3$ there is an Eisenstein series $G_{k}^{\psi,\varphi}\in M_{k}(N,\psi\varphi)$ (see the definition on \cite[p. 127]{first}). Using Hecke's trick, via analytic continuation one can also obtain $G_2^{\psi,\varphi}\in M_{2}(N,\psi\varphi)$ when $\psi$ and $\varphi$ are not both trivial.

The Fourier expansions of $G_{k}^{\psi,\varphi}$ are given in the following Theorem \ref{thm:eisen}, which can be found in \cite[Theorem 4.5.1]{first}.
\begin{thm}\label{thm:eisen}
Let $k\in\N$ and characters $\psi$ and $\varphi$ of moduli $u$ and $v$, respectively, be given such that either $k\geq 3$ or ($k=2$ and at least one of $\psi$ or $\varphi$ is non-trivial). The Eisenstein series $G_k^{\psi,\varphi}\in M_{k}(N,\psi\varphi)$ takes the form
\begin{equation}\label{eqn:eisen}
G_k^{\psi,\varphi}=\frac{C_kg(\overline{\varphi})}{v^k}E^{\psi,\varphi}_k(\tau),
\end{equation}
where $C_k\in\C$ is a constant depending only on $k$ and $E^{\psi,\varphi}_k(\tau)$ has Fourier expression
\begin{equation}\label{eqn:efe}
E^{\psi,\varphi}_k(\tau)=\delta(\psi)L(1-k,\varphi)+2\displaystyle\sum_{n=1}^{\infty}\sigma_{k-1}^{\psi,\varphi}(n)q^n,\quad q=e^{2\pi i \tau}.
\end{equation}
Here $L(s,\varphi)$ is the $L$-function for the character $\varphi$, $\delta(\psi)$ is $1$ if $\psi =  {\bm{1}}_{1}$ and is  $0$ otherwise, and the generalized power sum in the Fourier coefficient is 
\begin{equation}\label{eqn:gps}
\sigma_{k-1}^{\psi,\varphi}(n)=\displaystyle\sum_{\substack{m\mid n\\m>0}}\psi(n/m)\varphi(m)m^{k-1}.
\end{equation}
\end{thm}
Following \cite[p.\hskip 0.1cm 129]{first}, for any positive integer $N$ and any integer $k\geqslant 3$, let $A_{N,k}$ be the set of triples $(\psi,\varphi,t)$ such that $\psi$ and $\varphi$ are primitive Dirichlet characters modulo $u$ and $v$ with $(\psi\varphi)(-1) = (-1)^k$, and $t$ is a positive integer such that $tuv\mid N$. Then $|A_{N,k}|=\dim(\mathcal{E}_k(\varGamma_1(N)))$. For any triple $(\psi,\varphi,t)\in A_{N,k}$ define
\[
E_{k}^{\psi,\varphi,t}(\tau)=E_{k}^{\psi,\varphi}(t\tau).
\]
For $k=2$, we follow \cite[p.\hskip 0.1cm 132 and p.\hskip 0.1cm133]{first} and let $A_{N,2}$ be the set of triples $(\psi,\varphi,t)$ such that $\psi$ and $\varphi$ are primitive Dirichlet characters modulo $u$ and $v$ with $(\psi\varphi)(-1)=1$, and $t$ is and integer such that $1<tuv\mid N$. Note that the triple $(\bm{1}_1,\bm{1}_1,1)$ is excluded from $A_{N,2}$. Then $|A_{N,2}|=$ dim$(\mathcal{E}_2(\varGamma_1(N)))$. For any triple $(\psi,\varphi,t)\in A_{N,2}$ define
\begin{align*}
E_{2}^{\psi,\varphi,t}(\tau)=
\begin{cases}
E_{2}^{\psi,\varphi}(t\tau)&\text{unless $\psi=\varphi=\bm{1}_1$},\\
E_{2}^{\bm{1}_1,\bm{1}_1}(\tau)-tE_{2}^{\bm{1}_1,\bm{1}_1}(t\tau)&\text{if $\psi=\varphi=\bm{1}_1$},
\end{cases}
\end{align*}
where we abuse notation and also use the definition \eqref{eqn:efe} in the case when $\psi$ and $\varphi$ are trivial. Combining \cite[Theorem 4.5.2, p.\hskip 0.1cm 129 and Theorem 4.6.2, p.\hskip 0.1cm 133]{first}, we obtain bases for certain Eisenstein series subspaces.
\begin{thm}\label{thm:eispace}
Let N be a positive integer and $k \geqslant 2$. The set 
\[
\{E_k^{\psi,\varphi,t}:(\psi,\varphi,t)\in A_{N,k}\}
\]
represents a basis of $\mathcal{E}_k(\varGamma_1(N))$. For any character $\chi$ modulo $N$, the set
\[
\{E_k^{\psi,\varphi,t}:(\psi,\varphi,t)\in A_{N,k},\,\psi\varphi = \chi\}
\]
represents a basis of $\mathcal{E}_k(N,\chi)$.
\end{thm}

\subsection{Well-known results.}
There are some well-known formulas for the sums of four squares and six squares which follow from the fact that the theta functions are indeed modular forms (which we have explained in an earlier part of the preliminaries section). A modern proof of Jacobi's formula \eqref{eqn:Jacobi} can be found in \cite{Jac}.
\begin{thm}[Jacobi, 1829]\label{thm:jacobi}
 For $m \in \mathbb{N}$, we have
\begin{equation}\label{eqn:jaco}
r_{(1,1,1,1)}(n) = 8\cdot\sigma(n)-32\cdot\delta_{4\mid n}\cdot\sigma\left(\frac{n}{4}\right).
\end{equation}
\end{thm}

Letting $\left(\frac{\cdot}{\cdot}\right)$ denote the \textbf{Kronecker-Jacobi-Legendre symbol}, a similar formula for representations of integers as sums of six squares can be found in \cite[Sum of 6 Squares, p.\hskip 0.1cm 38]{Six}.
\begin{thm}\label{thm:six}
For $m \in \mathbb{N}$, we have
\begin{equation}\label{eqn:jacobi}
r_{(1,1,1,1,1,1)}(n) = 16\cdot\displaystyle\sum_{m\mid n}{\chi_{-4}\left(\frac{n}{m}\right)m^2}-4\cdot\displaystyle\sum_{m\mid n}{\chi_{-4}(m)m^2}.
\end{equation}
\end{thm}

\section{The proof of Theorem \ref{thm:relationship}}\label{sec:relationshipthm}

\begin{proof}[Proof of Theorem \ref{thm:relationship}]
(1) For the equation
\[\displaystyle\sum_{i=1}^{k}{a_i x_i^2}=n,
\]
we classify the solutions by the greatest common divisor of $x_1,x_2,\dots,x_k$. Setting
\begin{align*}
A(n)&:=\left\{(x_1,x_2,\cdots, x_k)\in\mathbb{Z}^k:\,\displaystyle\sum_{i=1}^{k}{a_i x_i^2}=n\right\},\\
A_d(n)&:=\{(x_1,x_2,\cdots, x_k)\in A(n):\, {\rm gcd}(x_1,x_2,\cdots ,x_k)=d,\, d\in \mathbb{N}\},
\end{align*}
we have a natural splitting of $A(n)$ into a disjoint union
\[
\bigcup_{d^2\mid n}^{\cdot}A_d(n) =A(n),
\]
which implies that 
\[
\displaystyle\sum_{d^2\mid n}{|A_d(n)|}=|A(n)|.
\]
For any $(x_1,x_2,\cdots, x_k)\in A_d(n)$, The following two equations
\[{\rm gcd}\left(\frac{x_1}{d},\cdots ,\frac{x_k}{d}\right)=1,
\]
\[\displaystyle\sum_{i=1}^{k}{a_i\left(\frac{x_i}{d}\right)^2}=\frac{n}{d^2},
\]
hold. This yields
\[\left(\frac{x_1}{d},\frac{x_2}{d},\cdots, \frac{x_k}{d}\right)\in A_1\left(\frac{n}{d^2}\right).
\]
Conversely, if $(y_1,y_2,\cdots, y_k)\in A_1\left(\frac{n}{d^2}\right)$, then
\[
(dy_1,dy_2,\cdots, dy_k)\in A_d(n).
\]
Hence, there is a bijection between the elements of $A_d(n)$ and $A_1\left(\frac{n}{d^2}\right)$. We conclude that
\[|A_d(n)|=\left|A_1\left(\frac{n}{d^2}\right)\right|=r_{\bm{a}}^{p}\left(\frac{n}{d^2}\right).
\]
Hence, we get the equation
\[r_{\bm{a}}(n)=|A(n)|=\displaystyle\sum_{d^2\mid n}{|A_d(n)|}=\displaystyle\sum_{d^2\mid n}{r_{\bm{a}}^p\left(\frac{n}{d^2}\right)}.
\]
\noindent (2) Now let us prove the second equation by the inclusion-exclusion principle.
Assume that the prime factorization of $n$ is $n=p_1^{\alpha_1}p_2^{\alpha_2}\dots p_t^{\alpha_t}$. Denote
\[B_i=\{(x_1,x_2,\dots,x_k)\in\mathbb{Z}^k:\,p_i\mid gcd(x_1,x_2,\dots,x_k)\}\qquad \forall i\in \{1,2,\dots ,t\}.
\]
By the inclusion-exclusion principle, we get
\begin{align*}
&|B_1\cup B_2\cup \dots \cup B_t|\\
=&\displaystyle\sum_{i}|B_i|-\displaystyle\sum_{i<j}|B_i\cap B_j|+\dots+(-1)^{t-1}|B_1\cap B_2 \cap \dots \cap B_t|\\
=&\displaystyle\sum_{p_i^2\mid n}{r_{\bm{a}}\left(\frac{n}{p_i^2}\right)}-\displaystyle\sum_{p_i^2p_j^2\mid n}{r_{\bm{a}}\left(\frac{n}{p_i^2p_j^2}\right)}+\dots +(-1)^{t-1}\displaystyle\sum_{p_1^2p_2^2\dots p_t^2\mid n}{r_{\bm{a}}\left(\frac{n}{p_1^2p_2^2\dots p_t^2}\right)}.&
\end{align*}
This yields
\begin{multline}
\label{eqn:musum} r_{\bm{a}}^p(n)=r_{\bm{a}}(n)-|B_1\cup B_2\cup \dots \cup B_t|\\
=r_{\bm{a}}(n)-\displaystyle\sum_{p_i^2\mid n}{r_{\bm{a}}\left(\frac{n}{p_i^2}\right)}+\displaystyle\sum_{p_i^2p_j^2\mid n}{r_{\bm{a}}\left(\frac{n}{p_i^2p_j^2}\right)}-\dots +(-1)^{t}\displaystyle\sum_{p_1^2p_2^2\dots p_t^2\mid n}{r_{\bm{a}}\left(\frac{n}{p_1^2p_2^2\dots p_t^2}\right)}\\
=\sum_{S\subseteq\{1,\dots,t\}} (-1)^{\#S} r_{\bm{a}}\left(\frac{n}{\prod_{j\in S} p_j^2}\right).
\end{multline}
Note that for $d^2\mid n$, we have 
\[
\mu(d)=\begin{cases} (-1)^{\#S}&\text{if }d=\prod_{j\in S} p_{j}\ S\subseteq\{1,\dots,t\},\\
0&\text{otherwise},
\end{cases}
\]
so \eqref{eqn:musum} may be rewritten as 
\[
r_{\bm{a}}^p(n)=\sum_{d^2\mid n}\mu(d)r_{\bm{a}}\left(\frac{n}{d^2}\right).
\]
\end{proof}
\begin{re}
{\rm Each equation can be derived from the other one by} Möbius Inversion.\\
~\\
{\rm (a) }Suppose that $r_{\bm{a}}(n)=\displaystyle\sum_{d^2|n}{r^{p}_{\bm{a}}\left(\frac{n}{d^2}\right)}\,$ is given. Then
\begin{align*}
\displaystyle\sum_{d^2\mid n}{\mu(d)r_{\bm{a}}\left(\frac{n}{d^2}\right)}&=\displaystyle\sum_{d^2\mid n}{\mu(d)}\displaystyle\sum_{\ell^2\mid\frac{n}{d^2}}{r_{\bm{a}}^p\left(\frac{n}{(\ell d)^2}\right)}=\displaystyle\sum_{m^2\mid n}{r_{\bm{a}}^p\left(\frac{n}{m^2}\right)}\displaystyle\sum_{d\mid m}{\mu(d)}\\
&=\displaystyle\sum_{m^2\mid n}{r_{\bm{a}}^p\left(\frac{n}{m^2}\right)}I(m)=r_{\bm{a}}^p\left(\frac{n}{1^2}\right)I(1)=r_{\bm{a}}^p(n).&
\end{align*}
{\rm (b) }Suppose that $r_{\bm{a}}^{p}(n)=\displaystyle\sum_{d^2\mid n}{\mu(d)}{r_{\bm{a}}\left(\frac{n}{d^2}\right)}$ is given. Then
\begin{align*}
\displaystyle\sum_{d^2\mid n}{r_{\bm{a}}^p\left(\frac{n}{d^2}\right)}&=\displaystyle\sum_{d^2\mid n}\displaystyle\sum_{\ell^2\mid\frac{n}{d^2}}{\mu(\ell)r_{\bm{a}}\left(\frac{n}{(\ell d)^2}\right)}=\displaystyle\sum_{m^2\mid n}{r_{\bm{a}}\left(\frac{n}{m^2}\right)}\displaystyle\sum_{d\mid m}{\mu(d)}\\
&=\displaystyle\sum_{m^2\mid n}{r_{\bm{a}}\left(\frac{n}{m^2}\right)}I(m)=r_{\bm{a}}\left(\frac{n}{1^2}\right)I(1)=r_{\bm{a}}(n).&
\end{align*}
\end{re}

\section{The proof of Theorem \ref{thm:mainthm}}\label{sec:MainThm}
\begin{proof}[Proof of Theorem \ref{thm:mainthm}]
(1) We first prove the identity
\begin{multline}\label{eqn:one}
\displaystyle\sum_{d^2\mid n}{\mu (d)}\displaystyle\sum_{m\mid\frac{n}{d^2}}{\psi\left(\frac{n}{md^2}\right)\varphi(m)m^h}\\
=\displaystyle\prod_{u_s}{\left(1-\delta_{u_s\mid n}-\delta_{u_s^2\mid n}+\delta_{u_s^3\mid n}\right)}\displaystyle\sum_{d^2\mid n_3}{\mu (d)}\displaystyle\sum_{m\mid\frac{n_3}{d^2}}{\psi\left(\frac{n_3}{md^2}\right)\varphi(m)m^h}.
\end{multline}
Recall that 
\[
n=\displaystyle\prod_{u_s}{u_s^{c_s}}\cdot n_3,
\]
where the product over $u_s$ runs over the prime factors of $\gcd(N,M)$, with $N$ denoting the modulus of $\psi$ and $M$ denoting the modulus of $\varphi$. We prove \eqref{eqn:one} by induction on $\sum_{s}c_s$. If $\sum_{s}c_s=0$, then \eqref{eqn:one} is trivial. Otherwise, there is a prime $u_s\mid n$ and for $d^2\mid n$ and $m\mid \frac{n}{d^2}$ we have 
\[
\mu(d)\psi\left(\frac{n}{md^2}\right)\varphi(m)=0
\]
unless $0\leq \ord_{u_s}(d)\leq 1$, $\ord_{u_s}(m)=0$, and $\ord_{u_s}\left(\frac{n}{d^2}\right)=0$. If $c_s$ is odd or $c_s\geq 3$, then $0\leq \ord_{u_s}(d)\leq 1$, and $\ord_{u_s}\left(\frac{n}{d^2}\right)=0$ cannot simultaneously hold, and the entire sum vanishes; the factor on the right-hand side of \eqref{eqn:one} also vanishes in these cases. Finally, if $c_s=2$, then we must have $\ord_{u_s}(d)=1$ and $m\mid \frac{n}{u_s^2d'}$ for $d'=\frac{d}{u_s}$. Noting that $\mu(d)=-\mu(d')$, the left-hand side of \eqref{eqn:one} then becomes
\[
-\sum_{d'^2\mid \frac{n}{u_s^2}}{\mu (d')}\displaystyle\sum_{m\mid\frac{n}{u_s^2 d'^2}}{\psi\left(\frac{n}{u_s^2 md'^2}\right)\varphi(m)m^h}.
\]
The factor $-1$ in front is precisely $1-\delta_{u_s\mid n}-\delta_{u_s^2\mid n}$ and the remaining sum is the sum for $n\mapsto \frac{n}{u_s^2}$, and \eqref{eqn:one} follows by induction. 

To evaluate the remaining sum on the right-hand side of \eqref{eqn:one}, we next claim that 
\begin{multline}\label{eqn:two}
\displaystyle\sum_{d^2\mid n_3}{\mu (d)}\displaystyle\sum_{m\mid\frac{n_3}{d^2}}{\psi\left(\frac{n_3}{md^2}\right)\varphi(m)m^h}\\
=\displaystyle\prod_{q_j}{\left(1-\delta_{q_j^2\mid n}\overline{\psi(q_j)}^2\right)}\displaystyle\prod_{q_j}{\psi\left(q_j^{\nu_j}\right)}\displaystyle\sum_{d^2\mid n_2}\mu(d)\displaystyle\sum_{m\mid\frac{n_2}{d^2}}{\psi\left(\frac{n_2}{md^2}\right)\varphi(m)m^h},
\end{multline}
where we recall that 
\[ 
n_3=\displaystyle\prod_{q_j}{q_j^{\nu_j}}\cdot n_2,
\]
with $q_j\mid M$ running through the prime factors of $M$ which do not divide $N$. Once again, we have 
\[
\mu(d)\varphi(m)=0
\]
unless $0\leq \ord_{q_j}(d)\leq 1$ and $\ord_{q_j}(m)=0$ for every $j$. The condition $\ord_{q_j}(m)=0$ and the total multiplicativity of $\psi$ and $q_j\nmid N$ gives 
\begin{multline*}
\sum_{d^2\mid n_3}{\mu (d)}\displaystyle\sum_{m\mid\frac{n_3}{d^2}}{\psi\left(\frac{n_3}{md^2}\right)\varphi(m)m^h}\\
=\sum_{d^2\mid n_3}{\mu (d)}\psi(q_j)^{\nu_j-2\ord_{q_j}(d)}\displaystyle\sum_{m\mid\frac{n_3}{q_j^{\nu_j}d^2}}{\psi\left(\frac{n_3}{q_j^{\nu_j-2\ord_{q_j}(d)}md^2}\right)\varphi(m)m^h}.
\end{multline*}
If $q_j^2\nmid n$, then we necessarily have $\ord_{q_j}(d)=0$ and $d^2\mid n_3$ is equivalent to $d^2\mid \frac{n_3}{q_j^{\nu_j}}$, so we are done by induction. If $q_j^2\mid n$, then we write $d=d'q_j^{\alpha}$ for $0\leq \alpha\leq 1$ and simplify (noting that $\mu(d)=(-1)^{\alpha}\mu(d')$)
\[
\psi(q_j)^{\nu_j}\sum_{\alpha=0}^1 (-1)^{\alpha}\psi(q_j)^{-2\alpha}\sum_{d'^2\mid \frac{n_3}{q_j^{\nu_j}}}{\mu (d')}\displaystyle\sum_{m\mid\frac{n_3}{q_j^{\nu_j}d^2}}{\psi\left(\frac{n_3}{q_j^{\nu_j}md'^2}\right)\varphi(m)m^h}.
\]
The sum over $\alpha$ becomes
\[
1-\overline{\psi(q_j)}^2,
\]
and \eqref{eqn:two} follows by induction.

To evaluate the remaining sum in \eqref{eqn:two}, we prove that
\begin{multline}\label{eqn:three}
\displaystyle\sum_{d^2\mid n_2}\mu(d)\displaystyle\sum_{m\mid\frac{n_2}{d^2}}{\psi\left(\frac{n_2}{md^2}\right)\varphi(m)m^h}\\
=\varphi\left(\displaystyle\prod_{p_i}{p_i^{\gamma_i}}\right)\left(\displaystyle\prod_{p_i}{p_i^{\gamma_i}}\right)^h\displaystyle\prod_{p_i^2\mid n}{\left(1-\frac{\overline{\varphi(p_i)}^2}{p_i^{2h}}\right)}\displaystyle\sum_{m\mid n_1}{\psi\left(\frac{n_1}{m}\right)\varphi(m)\,m^h}\displaystyle\prod_{\substack{p^2|\frac{n_1}{m}\\p\,\, prime}}{\left(1-\overline{\psi(p^2)}\right)},
\end{multline}
where we recall that
\[
n_2=\displaystyle\prod_{p_i}{p_i^{\gamma_i}}\cdot n_1
\]
with $p_i$ running over the prime divisors of $N$ which do not divide $M$. Since 
\[
\mu(d)\psi\left(\frac{n_2}{md^2}\right)=0
\]
unless $0\leq \ord_{p_i}(d)\leq 1$ and $\ord_{p_i}(\frac{n_2}{md^2})=0$ for every $i$, we may write $d=d'p_i^{\alpha}$ with $0\leq \alpha\leq \min(\lfloor\frac{\gamma_i}{2}\rfloor,1)$, $d'\mid \frac{n_2}{p_i^{\gamma_i}}$, and 
\[
m= p_i^{\gamma_i-2\alpha} m'
\]
with $m'\mid \frac{n_2}{p_i^{\gamma_i} d'}$. This gives 
\begin{multline*}
\sum_{d^2\mid n_2}\mu(d)\displaystyle\sum_{m\mid\frac{n_2}{d^2}}{\psi\left(\frac{n_2}{md^2}\right)\varphi(m)m^h}\\
=\sum_{\alpha=0}^{\min(\lfloor\frac{\gamma_i}{2}\rfloor,1)}  (-1)^{\alpha}\overline{\varphi(p_i)}^{2\alpha} p_i^{-2\alpha h}\varphi\left(p_i^{\gamma_i}\right)p_i^{h\gamma_i} 
\sum_{d'^2\mid \frac{n_2}{p_i^{\gamma_i}}}\mu(d')\displaystyle\sum_{m'\mid\frac{n_2}{p_i^{\gamma_i}d'^2}}{\psi\left(\frac{n_2}{p_i^{\gamma_i}m'd'^2}\right)\varphi(m')m'^h}.
\end{multline*}
The sum over $\alpha$ simplifies as 
\[
1-\delta_{p_i^2\mid n}\frac{\overline{\varphi(p_i)}^2}{p_i^{2h}}.
\]
By induction, we obtain 
\begin{multline}\label{eqn:threealmost}
\sum_{d^2\mid n_2}\mu(d)\displaystyle\sum_{m\mid\frac{n_2}{d^2}}{\psi\left(\frac{n_2}{md^2}\right)\varphi(m)m^h}\\
=\varphi\left(\prod_{p_i} p_i^{\gamma_i}\right)\left(\prod_{p_i} p_i^{\gamma_i}\right)^h\displaystyle\prod_{p_i^2\mid n}{\left(1-\frac{\overline{\varphi(p_i)}^2}{p_i^{2h}}\right)}\sum_{d\mid n_1}\mu(d)\sum_{m\mid \frac{n_1}{d^2}} \psi\left(\frac{n_1}{md^2}\right)\varphi(m)m^h.
\end{multline}
Using 
\[
\psi\left(\frac{n_1}{md^2}\right)=\psi\left(\frac{n_1}{m}\right) \overline{\psi(d^2)},
\]
we may interchange the sum over $m$ and $d$ to obtain that the remaining double sum in \eqref{eqn:threealmost} becomes
\[
\sum_{m\mid n_1} \psi\left(\frac{n_1}{m}\right)\varphi(m)m^h\sum_{d^2\mid \frac{n_1}{m}} \mu(d)\overline{\psi(d^2)}.
\]
By multiplicativity, we have 
\[
\sum_{d^2\mid \frac{n_1}{m}} \mu(d)\overline{\psi(d^2)}=\prod_{p^2\mid \frac{n_1}{m}}\left(1-\overline{\psi(p^2)}\right),
\]
yielding \eqref{eqn:three}.

Using the fact that $\psi(\frac{n_1}{m})=\psi(n_1)\overline{\psi(m)}$ for $m\mid n_1$ with $\gcd(m_1,N)=1$, we have 
\begin{multline*}
\displaystyle\sum_{m\mid n_1}{\psi\left(\frac{n_1}{m}\right)\varphi(m)m^h}\displaystyle\prod_{\substack{p^2\mid \frac{n_1}{m}\\p\,\, prime}}{\left(1-\overline{\psi(p)}^2\right)}
=\psi(n_1)\sum_{m\mid n_1} \varphi(m)\overline{\psi(m)} m^h \displaystyle\prod_{\substack{p^2\mid \frac{n_1}{m}\\p\,\, prime}}{\left(1-\overline{\psi(p)}^2\right)}\\
=\psi(n_1)\prod_{p\mid n_1} \sum_{j=0}^{\ord_p(n_1)} \left(\varphi(p)\overline{\psi(p)} p^{h}\right)^j\left(1-\delta_{j\leq \ord_p(n_1)-2}\overline{\psi(p)}^2\right),
\end{multline*}
where we used multiplicativity in the last step. Splitting the sum and evaluating the geometric series, the sum over $j$ simplifies as 
\begin{multline}\label{eqn:four}
\sum_{j=0}^{\ord_p(n_1)} \left(\varphi(p)\overline{\psi(p)} p^{h}\right)^j\left(1-\delta_{j\leq \ord_p(n_1)-2}\overline{\psi(p)}^2\right)
\\
=\left(1-\overline{\psi(p)}^2\right)\sum_{j=0}^{\ord_p(n_1)-2} \left(\varphi(p)\overline{\psi(p)} p^{h}\right)^j + \left(\varphi(p)\overline{\psi(p)}p^{h}\right)^{\ord_p(n_1)}\left(1+\frac{\psi(p)\overline{\varphi(p)}}{p^h}\right)\\
=\delta_{\ord_p(n_1)\geq 2}\left(1-\overline{\psi(p)}^2\right)\frac{1-\left(\varphi(p)\overline{\psi(p)}p^h\right)^{\ord_p(n_1)-1}}{1-\varphi(p)\overline{\psi(p)}p^h}+ \left(\varphi(p)\overline{\psi(p)}p^{h}\right)^{\ord_p(n_1)}\left(1+\frac{\psi(p)\overline{\varphi(p)}}{p^h}\right).
\end{multline}
Plugging \eqref{eqn:four} into \eqref{eqn:three}, then plugging this into \eqref{eqn:two}, and finally plugging into \eqref{eqn:one} yields (recalling that $n_1=\prod_{r_k\mid n_1} r_k^{\lambda_k}$ and $n_2=n_1\prod_{p_i} p_i^{\gamma_i}$)
\begin{multline*}
\sum_{d^2\mid n} \mu(d)\sum_{m\mid \frac{n}{d^2}} \psi\left(\frac{n}{md^2}\right)\varphi(m)m^h=\varphi(n_2) n_2^h
\prod_{u_s}\left(1-\delta_{u_s\mid n}-\delta_{u_s^2\mid n} + \delta_{u_s^3\mid n}\right)\\
\times 
\prod_{q_{j}}\left(1-\delta_{q_j^2\mid n} \overline{\psi(q_j)}^2\right)\psi\left(q_j^{\nu_j}\right)\prod_{p_i^2\mid n}\left(1-\frac{\overline{\varphi(p_i)}^2}{p_i^{2h}}\right) \\
\times \prod_{r_k\mid n_1}\left(\frac{\delta_{r_k^2\mid n}\left(1-\overline{\psi(r_k)}^2\right)}{\varphi(r_k^{\lambda_k})\overline{\psi(r_k^{\lambda_k})}r_k^{h\lambda_k}}\cdot \frac{1-\left(\varphi(r_k)\overline{\psi(r_k)} r_k^{h}\right)^{\lambda_k-1}}{1-\varphi(r_k)\overline{\psi(r_k)} r_k^{h}} + \left(1+\frac{\psi(r_k)\overline{\varphi(r_k)}}{r_k^h}\right)\right).
\end{multline*}
This is the claim for generic characters.

\noindent (2) 
In the special case that $\psi$ is a real (quadratic) character, for $p\mid n_1$ we have $1-\overline{\psi(p)}^2=0$, so the product over $r_k\mid n_1$ from part (1) becomes 
\[
\prod_{r_k\mid n_1} \left(1+\frac{\psi(r_k)\overline{\varphi(r_k)}}{r_k^h}\right).
\]
For $p\mid \frac{n}{n_1}$, we have $p\mid NM$, so either $\psi(p)=0$ or $\varphi(p)=0$, and hence 
\[
\prod_{r_k\mid n_1} \left(1+\frac{\psi(r_k)\overline{\varphi(r_k)}}{r_k^h}\right)=\prod_{r_k\mid n} \left(1+\frac{\psi(r_k)\overline{\varphi(r_k)}}{r_k^h}\right).
\]
Plugging in $\psi(q_j)^2=1$ yields the claim.

\end{proof}

\section{Examples}\label{sec:example}
In this section, we consider a number of problems related to representations of integers by quadratic forms where Theorem \ref{thm:mainthm} may be used with both $\psi$ and $\varphi$ real. Using the Kronecker symbol, we denote the real Dirichlet character by $\chi_n(m):=\left(\frac{n}{m}\right)$.
\begin{enumerate}[leftmargin=*,align=left,label={\rm(\arabic*)}]
\item For representations $x_1^2+x_2^2+x_3^2+x_4^2 = n$ as sums of four squares, Jacobi's formula \eqref{eqn:Jacobi} gives 
\[
r_{(1,1,1,1)}(n)=8\cdot\sigma(n)-32\cdot\delta_{4\mid n}\cdot\sigma\left(\frac{n}{4}\right).
\]
By Theorem \ref{thm:relationship}, we have:
\begin{align*}
r^{p}_{(1,1,1,1)}(n)&=\displaystyle\sum_{d^2\mid n}{\mu(d)\cdot r_{(1,1,1,1)}\left(\frac{n}{d^2}\right)}\\
&=8\cdot\displaystyle\sum_{d^2\mid n}{\mu(d)\cdot\displaystyle\sum_{m\mid\frac{n}{d^2}}m}-32\cdot\displaystyle\sum_{d^2|n}{\mu(d)\cdot\delta_{4\mid\frac{n}{d^2}}\cdot\displaystyle\sum_{m\mid\frac{n}{4d^2}}m}.&
\end{align*}
Let $\psi$ and $\varphi$ be $\bm{1}_1$. By Theorem \ref{thm:mainthm} (2), we have
\begin{align*}
&\displaystyle\sum_{d^2\mid n}{\mu(d)\cdot\displaystyle\sum_{m\mid\frac{n}{d^2}}m}=n\cdot\displaystyle\prod_{p}\left(1+\frac{1}{p}\right),\\
&\displaystyle\sum_{d^2\mid n}{\mu(d)\cdot\delta_{4\mid\frac{n}{d^2}}\cdot\displaystyle\sum_{m\mid\frac{n}{4d^2}}m}=\delta_{4\mid n}\cdot\left(\frac{n}{4}\right)\cdot\displaystyle\prod_{p\mid\frac{n}{4}}{\left(1+\frac{1}{p}\right)}.
\end{align*}
Thus
\[r^{p}_{(1,1,1,1)}(n)=8\cdot\left(1+\frac{1}{2}\delta_{2\mid n}-\delta_{4\mid n}-\frac{1}{2}\delta_{8\mid n}\right)\cdot n\displaystyle\prod_{p\ne 2}\left(1+\frac{1}{p}\right).
\]

\item Consider representations $x_1^2+x_2^2+x_3^2+x_4^2+x_5^2+x_6^2= n$ of $n$ as a sum of $6$ squares.\vspace{2ex}\\
Assume that $\bm a = (1,1,1,1,1,1)$ and $n=2^{\alpha}p_1^{\alpha_1}\cdots p_g^{\alpha_g}.$ It is well-known that 
\[r_{\bm a}(n)=16\cdot\displaystyle\sum_{m\mid n}{\chi_{-4}\left(\frac{n}{m}\right)m^2}-4\cdot\displaystyle\sum_{m\mid n}{\chi_{-4}(m)m^2}.
\]
By Theorem \ref{thm:relationship}, we have:
\begin{align*}
r^{p}_{\bm a}(n)&=\displaystyle\sum_{d^2\mid n}{\mu(d)\cdot r_{\bm a}\left(\frac{n}{d^2}\right)}\\
&=16\cdot\displaystyle\sum_{d^2\mid n}{\mu(d)\displaystyle\sum_{m\mid\frac{n}{d^2}}\chi_{-4}\left(\frac{n}{md^2}\right)m^2}-4\cdot\displaystyle\sum_{d^2\mid n}{\mu(d)\displaystyle\sum_{m\mid\frac{n}{d^2}}{\chi_{-4}(m)m^2}}.&
\end{align*}
Let $\varphi$ be $\bm{1}_1$ and $\psi$ be $\chi_{-4}$. In this case we have $n=n_2$ and Theorem \ref{thm:mainthm} (2) gives
\[\displaystyle\sum_{d^2\mid n}{\mu(d)\displaystyle\sum_{m|\frac{n}{d^2}}\chi_{-4}\left(\frac{n}{md^2}\right)m^2}=\left(1-\frac{\delta_{4\mid n}}{16}\right)\cdot n^2\displaystyle\prod_{p\mid n}\left(1+\frac{\chi_{-4}(p)}{p^2}\right).
\]
Let $\psi$ be $\bm{1}_1$ and $\varphi$ be $\chi_{-4}$. In this case, we have $n_2$ is the odd part of $n$, and letting $\alpha=\ord_2(n)$, Theorem \ref{thm:mainthm} (2) gives
\[
\displaystyle\sum_{d^2\mid n}{\mu(d)\displaystyle\sum_{m\mid\frac{n}{d^2}}\chi_{-4}(m)m^2}=\left(1-\delta_{4|n}\right)\chi_{-4}\left(\frac{n}{2^{\alpha}}\right)\cdot \left(\frac{n}{2^{\alpha}}\right)^2\displaystyle\prod_{p\mid n}\left(1+\frac{\chi_{-4}(p)}{p^2}\right).
\]
Then
\[
r^{p}_{\bm a}(n)=d(n)\cdot n^2\displaystyle\prod_{p\mid n}\left(1+\frac{\chi_{-4}(p)}{p^2}\right),
\]
where
\begin{align*}
d(n):=16-\delta_{4\mid n}-4^{1-\alpha}(1-\delta_{4\mid n})\chi_{-4}\left(\frac{n}{2^{\alpha}}\right)=
\begin{cases}
12&\text{if $n$ $\equiv$ $1$ (mod $4$)},\\
20&\text{if $n$ $\equiv$ $3$ (mod $4$)},\\
15&\text{if $n$ $\equiv$ $0$, $2$, $4$ (mod $8$)},\\
17&\text{if $n$ $\equiv$ $6$ (mod $8$)}.\\
\end{cases}
\end{align*}
\end{enumerate}

\section*{Acknowledgements}
The authors thank the referee for helpful comments about the exposition. The second author would like to express thanks to his research fellow Liu Yiming, who provided the formulas for $r_{\bm{a}}(n)$ in the appendix. These formulas helped to give more examples of the application of Theorem \ref{thm:relationship} and Theorem \ref{thm:mainthm}. Examples are given in the appendix below.

\appendix
\section{Some formulas derived by the two theorems}\label{sec:app}
Once obtained numerically, the formulas below directly follow from the valence formula [see \cite[Proposition 8, p.\hskip 0.1cm 115]{koblitz_introduction_1993} for more details], Theorem \ref{thm:eisen} and Theorem \ref{thm:eispace}. We only need to check sufficiently many coefficients to make sure that the formula provided by Liu Yiming is correct. This can be done by simple linear algebra.
\begin{enumerate}[leftmargin=*,align=left,label={\rm(\arabic*)}]
\item Consider representations $\displaystyle\sum_{i=1}^8{x_i^2}=n$ of $n$ as sums of $8$ squares. Then we have
$$r_{(1,1,1,1,1,1,1,1)}(n)=16\cdot\sigma_3(n)-32\cdot\delta_{2\mid n}\cdot\sigma_3\left(\frac{n}{2}\right)+256\cdot\delta_{4\mid n}\cdot\sigma_3\left(\frac{n}{4}\right).$$
In this case, Theorem \ref{thm:mainthm} (2) gives
$$r^{p}_{(1,1,1,1,1,1,1,1)}(n)=\left(16-2\delta_{2\mid n}+\frac{7}{2}\delta_{4\mid n}+\frac{1}{2}\delta_{8\mid n}\right)\cdot n^3\displaystyle\prod_{p\ne 2}{\left(1+\frac{1}{p^3}\right)}.$$

\item Consider representations by the quadratic form $x_1^2+x_2^2+x_3^2+2x_4^2 = n$. The valence formula gives
$$r_{(1,1,1,2)}(n)=-2\cdot\displaystyle\sum_{m\mid n}{\chi_8(m)m}+8\cdot\displaystyle\sum_{m\mid n}{\chi_8\left(\frac{n}{m}\right)m}.$$
Then Theorem \ref{thm:mainthm} (2) yields
$$r_{(1,1,1,2)}^{p}(n)=d_2(n)\cdot n\displaystyle\prod_{p}{\left(1+\frac{\chi_8(p)}{p}\right)},$$
where
$$d_2(n) :=  \left\{ \begin{array}{rcl}
-2\chi_8(n)+8 & \mbox{for} & 2 \nmid n,\vspace{1ex}\\
-\chi_8\left(\frac{n}{2}\right)+8 & \mbox{for} & 2\mid n\mbox{ but }4 \nmid n, \vspace{1ex}\\
6 &\mbox{for} &4\mid n. \end{array}\right.$$

\item Consider representations by the quadratic form {$x_1^2+x_2^2+x_3^2+3x_4^2 = n$}. The valence formula gives
\begin{align*}
r_{(1,1,1,3)}(n)=&-\displaystyle\sum_{m\mid n}{\chi_{12}(m)m}+6\cdot\displaystyle\sum_{m\mid n}{\chi_{12}\left(\frac{n}{m}\right)m}\\
&+3\cdot\displaystyle\sum_{m\mid n}{\chi_{-3}\left(\frac{n}{m}\right)\chi_{-4}(m)m}-2\cdot\displaystyle\sum_{m\mid n}{\chi_{-4}\left(\frac{n}{m}\right)\chi_{-3}(m)m}.
\end{align*}
Assume that $n=2^{\alpha}\,3^{\beta}p^{\alpha_1}_1p^{\alpha_2}_2\cdots p^{\alpha_k}_k$. Then
$$r^{p}_{(1,1,1,3)}(n)=d_3(n)\cdot n\displaystyle\prod_{p}\left(1+\frac{\chi_{12}(p)}{p}\right),$$
where
$$d_3(n) :=  \left\{ \begin{array}{rcl}
\left(2+\frac{1}{2^{\alpha}}\chi_{-3}(2^{\alpha})\chi_{-4}(\frac{n}{2^{\alpha}})\right)\left(3-\frac{1}{3^{\beta}}\chi_{-4}(3^{\beta})\chi_{-3}(\frac{n}{3^{\beta}})\right) & \mbox{for} & \alpha,\beta \in \{0,1\},\vspace{1ex}\\
\frac{3}{2}\left(3-\frac{1}{3^{\beta}}\chi_{-4}(3^{\beta}\right)\chi_{-3}(\frac{n}{3^{\beta}})) & \mbox{for} & \alpha \geqslant 2, \beta \in\{0,1\},\vspace{1ex}\\
\frac{8}{3}\left(2+\frac{1}{2^{\alpha}}\chi_{-3}(2^{\alpha})\chi_{-4}(\frac{n}{2^{\alpha}})\right) &\mbox{for} & \beta\geqslant 2, \alpha\in\{0,1\},\vspace{1ex}\\
4 & \mbox{for} &\alpha \geqslant 2, \beta\geqslant 2.\end{array}\right.$$

\item Consider representations by the quadratic form {$x_1^2+x_2^2+x_3^2+4x_4^2 = n$}. The valence formula gives
$$r_{(1,1,1,4)}(n)=\left(2\cdot\chi_{-4}(n)+4\right)\cdot\sigma(n)-20\cdot\delta_{4\mid n}\cdot\sigma\left(\frac{n}{4}\right)+24\cdot\delta_{8\mid n}\cdot\sigma\left(\frac{n}{8}\right)-32\cdot\delta_{16\mid n}\cdot\sigma\left(\frac{n}{16}\right).$$
Then Theorem \ref{thm:mainthm} gives
$$r^{p}_{(1,1,1,4)}(n)=d_4(n)\cdot n\displaystyle\prod_{p\ne 2}{\left(1+\frac{1}{p}\right)},$$
where
\begin{multline*}
d_4(n):=\left(4+2\cdot\chi_{-4}(n)\right)+2\cdot\delta_{2\mid n}-\left(2\cdot\chi_{-4}(n)+\frac{1}{2}\cdot\chi_{-4}\left(\frac{n}{4}\right)+5\right)\cdot\delta_{4\mid n}\\
+\frac{1}{2}\cdot\delta_{8|n}-\frac{1}{2}\cdot\delta_{16|n}-\delta_{32|n}.
\end{multline*}

\item Consider representations by the quadratic form {$x_1^2+x_2^2+2x_3^2+2x_4^2 = n$}. The valence formula gives
$$r_{(1,1,2,2)}(n)=4\cdot\sigma(n)-4\cdot\delta_{2\mid n}\cdot\sigma\left(\frac{n}{2}\right)+8\cdot\delta_{4\mid n}\cdot\sigma\left(\frac{n}{4}\right)-32\cdot\delta_{8\mid n}\cdot\sigma\left(\frac{n}{8}\right).$$
Then Theorem \ref{thm:mainthm} (2) implies that 
$$r^{p}_{(1,1,2,2)}(n)=\left(4+\delta_{4\mid n}-3\cdot\delta_{8\mid n}-2\cdot\delta_{16\mid n}\right)\cdot n\displaystyle\prod_{p\ne 2}\left(1+\frac{1}{p}\right).$$

\item Consider representations by the quadratic form {$x_1^2+x_2^2+x_3^2+5x_4^2 = n$}. The valence formula gives
\begin{multline*}
r_{(1,1,1,5)}(n)=\displaystyle\sum_{m\mid n}{\chi_5(m)m}-2\cdot\delta_{2\mid n}\cdot\displaystyle\sum_{m\mid \frac{n}{2}}{\chi_5(m)m}-4\cdot\delta_{4|n}\cdot\displaystyle\sum_{m|\frac{n}{4}}{\chi_5(m)m}\\
+5\cdot\displaystyle\sum_{m\mid n}{\chi_5\left(\frac{n}{m}\right)m}+10\cdot\delta_{2\mid n}\cdot\displaystyle\sum_{m\mid\frac{n}{2}}{\chi_5\left(\frac{n}{2m}\right)m}-20\cdot\delta_{4\mid n}\cdot\displaystyle\sum_{m\mid \frac{n}{4}}{\chi_5\left(\frac{n}{4m}\right)m}.
\end{multline*}
Then Theorem \ref{thm:mainthm} (2) implies that 
$$r^{p}_{(1,1,1,5)}(n)=d_6(n)\cdot\left(1+\frac{1}{2}\delta_{2\mid n}-\frac{3}{2}\delta_{4\mid n}+\delta_{8\mid n}\right)\cdot n\displaystyle\prod_{p\ne 2}\left(1+\frac{\chi_5(p)}{p}\right),$$
where
$$d_6(n) :=  \left\{ \begin{array}{rcl}
\chi_5(n)+5 & \mbox{for} & 5 \nmid n,\vspace{1ex}\\
\frac{1}{5}\chi_5\left(\frac{n}{5}\right)+5 & \mbox{for} & 5\mid n\mbox{ but }25 \nmid n,\vspace{1ex}\\
\frac{24}{5} &\mbox{for} &25\mid n. \end{array}\right.$$

\item Consider representations by the quadratic form {$x_1^2+x_2^2+2x_3^2+3x_4^2 = n$}. Then the valence formula gives
\begin{align*}
r_{(1,1,2,3)}(n)=&-\frac{1}{3}\cdot\displaystyle\sum_{m\mid n}{\chi_{24}(m)m}+4\cdot\displaystyle\sum_{m\mid n}{\chi_{24}\left(\frac{n}{m}\right)m}\\
&-\displaystyle\sum_{m\mid n}{\chi_{-3}\left(\frac{n}{m}\right)\chi_{-8}(m)m}+\frac{4}{3}\cdot\displaystyle\sum_{m\mid n}{\chi_{-8}\left(\frac{n}{m}\right)\chi_{-3}(m)m}.
\end{align*}
Assume that $n=2^{\alpha}\,3^{\beta}p^{\alpha_1}_1p^{\alpha_2}_2\cdots p^{\alpha_k}_k$. Then Theorem \ref{thm:mainthm} (2) gives
$$r^{p}_{(1,1,2,3)}(n)=d_7(n)\cdot n\displaystyle\prod_{p}\left(1+\frac{\chi_{24}(p)}{p}\right).$$
where
$$d_7(n) :=  \left\{ \begin{array}{rcl}
\frac{1}{3}\left(4-\frac{1}{2^{\alpha}}\chi_{-3}(2^{\alpha})\chi_{-8}\left(\frac{n}{2^{\alpha}}\right)\right)\left(3+\frac{1}{3^{\beta}}\chi_{-8}(3^{\beta})\chi_{-3}\left(\frac{n}{3^{\beta}}\right)\right)& \mbox{for} & \alpha,\beta \in \{0,1\},\vspace{1ex}\\
\left(3+\frac{1}{3^{\beta }}\chi_{-8}(3^{\beta})\chi_{-3}\left(\frac{n}{3^{\beta}}\right)\right)  & \mbox{for} & \alpha \geqslant 2, \beta \in\{0,1\},\vspace{1ex}\\
\frac{8}{9}\left(4-\frac{1}{2^{\alpha}}\chi_{-3}(2^{\alpha})\chi_{-8}\left(\frac{n}{2^{\alpha}}\right)\right) &\mbox{for} & \beta\geqslant 2, \alpha\in\{0,1\},\vspace{1ex}\\
\frac{8}{3} & \mbox{for} &\alpha \geqslant 2, \beta\geqslant 2.\end{array}\right.$$

\item Consider representations by the quadratic form {$x_1^2+x_2^2+2x_3^2+4x_4^2 = n$}. By the valence formula, we have 
$$r_{(1,1,2,4)}(n)=-2\cdot\delta_{2\mid n}\cdot\displaystyle\sum_{m\mid\frac{n}{2}}{\chi_8(m)m}+4\cdot\displaystyle\sum_{m\mid n}{\chi_8\left(\frac{n}{m}\right)m}.$$
Then Theorem \ref{thm:mainthm} (2) gives 
$$r_{(1,1,2,4)}^{p}(n)=d_8(n)\cdot n\displaystyle\prod_{p}{\left(1+\frac{\chi_8(p)}{p}\right)},$$
where
$$d_8(n) :=  \left\{ \begin{array}{rcl}
4 & \mbox{for} & 2 \nmid n,\vspace{1ex}\\
4-\chi_8\left(\frac{n}{2}\right) & \mbox{for} & 2\mid n\mbox{ but }4 \nmid n, \vspace{1ex}\\
3-\frac{1}{2}\chi_8\left(\frac{n}{4}\right)&\mbox{for} &4\mid n\mbox{ but } 8\nmid n,\vspace{1ex}\\
3  &\mbox{for} & 8\mid n.\end{array}\right.$$

\item Consider representations by the quadratic form {$x_1^2+2x_2^2+2x_3^2+2x_4^2 = n$}. By  the valence formula, we have
$$r_{(1,2,2,2)}(n)=-2\cdot\displaystyle\sum_{m\mid n}{\chi_8(m)m}+4\cdot\displaystyle\sum_{m\mid n}{\chi_8\left(\frac{n}{m}\right)m}.$$
Then Theorem \ref{thm:mainthm} (2), we have 
$$r_{(1,2,2,2)}^{p}(n)=d_9(n)\cdot n\displaystyle\prod_{p}{\left(1+\frac{\chi_8(p)}{p}\right)},$$
where
$$d_9(n) :=  \left\{ \begin{array}{rcl}
-2\chi_8(n)+4 & \mbox{for} & 2 \nmid n,\vspace{1ex}\\
-\chi_8\left(\frac{n}{2}\right)+4 & \mbox{for} & 2\mid n\mbox{ but }4 \nmid n, \vspace{1ex}\\
3 &\mbox{for} &4\mid n. \end{array}\right.$$

\item Consider representations by the quadratic form {$x_1^2+x_2^2+3x_3^2+3x_4^2 = n$}. The valence formula yields
$$r_{(1,1,3,3)}(n)=4\cdot\sigma(n)-8\cdot\delta_{2\mid n}\cdot\sigma\left(\frac{n}{2}\right)-12\cdot\delta_{3\mid n}\cdot\sigma\left(\frac{n}{2}\right)$$
$$+16\cdot\delta_{4\mid n}\cdot\sigma\left(\frac{n}{4}\right)+24\cdot\delta_{6\mid n}\cdot\sigma\left(\frac{n}{6}\right)-48\cdot\delta_{12\mid n}\cdot\sigma\left(\frac{n}{12}\right).$$
Thus Theorem \ref{thm:mainthm} (2) implies htat
$$r^{p}_{(1,1,3,3)}(n)=4\cdot d_{10}(n)\cdot n\displaystyle\prod_{p\ne 2,3}\left(1+\frac{1}{p}\right),$$
where
$$d_{10}(n):=1-\frac{1}{2}\delta_{2|n}-\frac{2}{3}\delta_{3|n}+\frac{1}{2}\delta_{4|n}+\frac{1}{3}\delta_{6|n}+\frac{1}{2}\delta_{8|n}$$
$$-\frac{1}{3}\delta_{9|n}-\frac{1}{3}\delta_{12|n}+\frac{1}{6}\delta_{18|n}-\frac{1}{3}\delta_{24|n}-\frac{1}{6}\delta_{36|n}-\frac{1}{6}\delta_{72|n}.$$

\item Consider representations by the quadratic form {$x_1^2+x_2^2+2x_3^2+6x_4^2 = n$}. By the valence formula, we have 
$$r_{(1,1,2,6)}(n)=-\delta_{2\mid n}\cdot\displaystyle\sum_{m\mid\frac{n}{2}}{\chi_{12}(m)m}+3\cdot\displaystyle\sum_{m\mid n}{\chi_{12}\left(\frac{n}{m}\right)m}$$
$$+3\cdot\delta_{2\mid n}\cdot\displaystyle\sum_{m\mid\frac{n}{2}}{\chi_{-3}\left(\frac{n}{2m}\right)\chi_{-4}(m)m}+\displaystyle\sum_{m\mid n}{\chi_{-4}\left(\frac{n}{m}\right)\chi_{-3}(m)m}.$$
Assume that $n=2^{\alpha}\,3^{\beta}p^{\alpha_1}_1p^{\alpha_2}_2\cdots p^{\alpha_k}_k$. Then, by Theorem \ref{thm:mainthm} (2), we have 
$$r^{p}_{(1,1,2,6)}(n)=d_{11}(n)\cdot n\displaystyle\prod_{p}\left(1+\frac{\chi_{12}(p)}{p}\right),$$
where for $\alpha\in\{0,1,2\}$ and $\beta\in \{0,1\}$ we have 
\[
d_{11}(n):=\left(\left(1-\frac{1}{4}\delta_{4|n}\right)-\delta_{2|n}\frac{1}{2^{\alpha}}\chi_{-3}({2^{\alpha}})\chi_{-4}\left(\frac{n}{2^{\alpha}}\right)\right)\left(3+\frac{1}{3^{\beta}}\chi_{-4}(3^{\beta})\chi_{-3}\left(\frac{n}{3^{\beta}}\right)\right)
\]
and otherwise 
$$d_{11}(n) :=  \left\{ \begin{array}{rcl}
\frac{3}{4}\left(3+\frac{1}{3^{\beta}}\chi_{-4}(3^{\beta})\chi_{-3}\left(\frac{n}{3^{\beta}}\right)\right)& \mbox{for} & \alpha \geqslant 3, \beta \in\{0,1\},\vspace{1ex}\\
\frac{8}{3}\left(\left(1-\frac{1}{4}\delta_{4|n}\right)-\delta_{2|n}\frac{1}{2^{\alpha}}\chi_{-3}({2^{\alpha}})\chi_{-4}\left(\frac{n}{2^{\alpha}}\right)\right) &\mbox{for} & \beta\geqslant 2, \alpha\in\{0,1,2\},\vspace{1ex}\\
2 & \mbox{for} &\alpha \geqslant 3, \beta\geqslant 2.\end{array}\right.$$

\item Consider representations by the quadratic form {$x_1^2+2x_2^2+2x_3^2+3x_4^2 = n$}. The valence formula implies 
\begin{multline*}
r_{(1,2,2,3)}(n)=-\delta_{2\mid n}\cdot\displaystyle\sum_{m\mid\frac{n}{2}}{\chi_{12}(m)m}+3\cdot\displaystyle\sum_{m\mid n}{\chi_{12}\left(\frac{n}{m}\right)m}\\
-3\cdot\delta_{2\mid n}\cdot\displaystyle\sum_{m\mid\frac{n}{2}}{\chi_{-3}\left(\frac{n}{2m}\right)\chi_{-4}(m)m}-\displaystyle\sum_{m\mid n}{\chi_{-4}\left(\frac{n}{m}\right)\chi_{-3}(m)m}.
\end{multline*}
Assume that $n=2^{\alpha}\,3^{\beta}p^{\alpha_1}_1p^{\alpha_2}_2\cdots p^{\alpha_k}_k$. Then, by Theorem \ref{thm:mainthm} (2), we conclude that 
$$r^{p}_{(1,2,2,3)}(n)=d_{12}(n)\cdot n\displaystyle\prod_{p}\left(1+\frac{\chi_{12}(p)}{p}\right),$$
where for $\alpha\in\{0,1,2\}$ and $\beta\in\{0,1\}$ we have 
\[
d_{12}(n):=\left(\left(1-\frac{1}{4}\delta_{4|n}\right)+\delta_{2|n}\frac{1}{2^{\alpha}}\chi_{-3}({2^{\alpha}})\chi_{-4}\left(\frac{n}{2^{\alpha}}\right)\right)\left(3-\frac{1}{3^{\beta}}\chi_{-4}(3^{\beta})\chi_{-3}\left(\frac{n}{3^{\beta}}\right)\right)
\]
and otherwise
$$d_{12}(n) :=  \left\{ \begin{array}{rcl}
\frac{3}{4}\left(3-\frac{1}{3^{\beta}}\chi_{-4}(3^{\beta})\chi_{-3}\left(\frac{n}{3^{\beta}}\right)\right)& \mbox{for} & \alpha \geqslant 3, \beta \in\{0,1\},\vspace{1ex}\\
\frac{8}{3}\left(\left(1-\frac{1}{4}\delta_{4|n}\right)+\delta_{2|n}\frac{1}{2^{\alpha}}\chi_{-3}({2^{\alpha}})\chi_{-4}\left(\frac{n}{2^{\alpha}}\right)\right) &\mbox{for} & \beta\geqslant 2, \alpha\in\{0,1,2\},\vspace{1ex}\\
2 & \mbox{for} &\alpha \geqslant 3, \beta\geqslant 2.\end{array}\right.$$

\item Consider representations by the quadratic form {$x_1^2+2x_2^2+2x_3^2+4x_4^2 = n$}. The valence formula implies that
$$r_{(1,2,2,4)}(n)=2\cdot\sigma(n)-2\cdot\delta_{2|n}\cdot\sigma\left(\frac{n}{2}\right)+8\cdot\delta_{8|n}\cdot\sigma\left(\frac{n}{8}\right)-32\cdot\delta_{16|n}\cdot\sigma\left(\frac{n}{16}\right).$$
Then Theorem \ref{thm:mainthm} (2) implies that 
$$r^{p}_{(1,2,2,4)}(n)=\left(2-\frac{1}{2}\delta_{4|n}+\delta_{8|n}-\frac{3}{2}\delta_{16|n}-\delta_{32|n}\right)\cdot n\displaystyle\prod_{p\ne 2}\left(1+\frac{1}{p}\right).$$

\item Consider representations by the quadratic form{$x_1^2+2x_2^2+2x_3^2+6x_4^2 = n$}. By the valence formula, we have 
\begin{align*}
r_{(1,2,2,6)}(n)=&-\frac{1}{3}\cdot\displaystyle\sum_{m|n}{\chi_{24}(m)m}+2\cdot\displaystyle\sum_{m|n}{\chi_{24}\left(\frac{n}{m}\right)m}\\
&+\displaystyle\sum_{m|n}{\chi_{-3}\left(\frac{n}{m}\right)\chi_{-8}(m)m}-\frac{2}{3}\cdot\displaystyle\sum_{m|n}{\chi_{-8}\left(\frac{n}{m}\right)\chi_{-3}(m)m}.
\end{align*}
Assume that $n=2^{\alpha}\,3^{\beta}p^{\alpha_1}_1p^{\alpha_2}_2\cdots p^{\alpha_k}_k$. Then Theorem \ref{thm:mainthm} (2) implies that 
$$r^{p}_{(1,2,2,6)}(n)=d_{14}(n)\cdot n\displaystyle\prod_{p}\left(1+\frac{\chi_{24}(p)}{p}\right),$$
where
$$
d_{14}(n) :=  \left\{ \begin{array}{rcl}
\frac{1}{3}\left(2+\frac{1}{2^{\alpha}}\chi_{-3}(2^{\alpha})\chi_{-8}\left(\frac{n}{2^{\alpha}}\right)\right)\left(3-\frac{1}{3^{\beta}}\chi_{-8}(3^{\beta})\chi_{-3}\left(\frac{n}{3^{\beta}}\right)\right)& \mbox{for} & \alpha,\beta \in \{0,1\},\vspace{1ex}\\
\frac{1}{2}\left(3-\frac{1}{3^{\beta }}\chi_{-8}(3^{\beta})\chi_{-3}\left(\frac{n}{3^{\beta}}\right)\right)  & \mbox{for} & \alpha \geqslant 2, \beta \in\{0,1\},\vspace{1ex}\\
\frac{8}{9}\left(2+\frac{1}{2^{\alpha}}\chi_{-3}(2^{\alpha})\chi_{-8}\left(\frac{n}{2^{\alpha}}\right)\right) &\mbox{for} & \beta\geqslant 2, \alpha\in\{0,1\},\vspace{1ex}\\
\frac{4}{3} & \mbox{for} &\alpha \geqslant 2, \beta\geqslant 2.\end{array}\right.$$

\item Consider representations by the quadratic form {$x_1^2+2x_2^2+4x_3^2+4x_4^2 = n$}. By the valence formula, we have 
$$r_{(1,2,4,4)}(n)=-2\cdot\delta_{2|n}\cdot\displaystyle\sum_{m|\frac{n}{2}}{\chi_8(m)m}+2\cdot\displaystyle\sum_{m|n}{\chi_8\left(\frac{n}{m}\right)m}.$$
Thus Theorem \ref{thm:mainthm} (2) implies that 
$$r_{(1,2,4,4)}^{p}(n)=d_{15}(n)\cdot n\displaystyle\prod_{p}{\left(1+\frac{\chi_8(p)}{p}\right)},$$
where
$$d_{15}(n) :=  \left\{ \begin{array}{rcl}
2 & \mbox{for} & 2 \nmid n,\vspace{1ex}\\
2-\chi_8\left(\frac{n}{2}\right) & \mbox{for} & 2\mid n\mbox{ but }4 \nmid n, \vspace{1ex}\\
\frac{3}{2}-\frac{1}{2}\chi_8\left(\frac{n}{4}\right)&\mbox{for} &4\mid n\mbox{ but } 8\nmid n,\vspace{1ex}\\
\frac{3}{2}  &\mbox{for} & 8\mid n.\end{array}\right.$$

\item Consider representations by the quadratic form {$x_1^2+2x_2^2+4x_3^2+6x_4^2 = n$}. By the valence formula, we have
\begin{multline*}
r_{(1,2,4,6)}(n)=-\delta_{4|n}\cdot\displaystyle\sum_{m|\frac{n}{4}}{\chi_{12}(m)m}+\frac{3}{2}\cdot\displaystyle\sum_{m|n}{\chi_{12}\left(\frac{n}{m}\right)m}\\
-3\cdot\delta_{4|n}\cdot\displaystyle\sum_{m|\frac{n}{4}}{\chi_{-3}\left(\frac{n}{4m}\right)\chi_{-4}(m)m}+\frac{1}{2}\cdot\displaystyle\sum_{m|n}{\chi_{-4}\left(\frac{n}{m}\right)\chi_{-3}(m)m}.
\end{multline*}
Assume that $n=2^{\alpha}\,3^{\beta}p^{\alpha_1}_1p^{\alpha_2}_2\cdots p^{\alpha_k}_k$. Then
$$r^{p}_{(1,2,4,6)}(n)=d_{16}(n)\cdot n\displaystyle\prod_{p}\left(1+\frac{\chi_{12}(p)}{p}\right),$$
where for $\alpha\in \{0,1,2,3\}$ and $\beta\in\{0,1\}$ we have 
\[
d_{16}(n):=\left(\frac{1}{2}-\frac{1}{8}\delta_{4|n}-\delta_{4|n}\frac{1}{2^{\alpha}}\chi_{-3}(2^{\alpha})\chi_{-4}\left(\frac{n}{2^{\alpha}}\right)\right)\left(3+\frac{1}{3^{\beta}}\chi_{-4}(3^{\beta})\chi_{-3}\left(\frac{n}{3^{\beta}}\right)\right)
\]
and otherwise
$$d_{16}(n) :=  \left\{ \begin{array}{rcl}
\frac{3}{8}\left(3+\frac{1}{3^{\beta}}\chi_{-4}(3^{\beta})\chi_{-3}\left(\frac{n}{3^{\beta}}\right)\right)& \mbox{for} & \alpha \geqslant 4, \beta \in\{0,1\}\vspace{1ex},\\
\frac{8}{3}\left(\frac{1}{2}-\frac{1}{8}\delta_{4|n}-\delta_{4|n}\frac{1}{2^{\alpha}}\chi_{-3}({2^{\alpha}})\chi_{-4}\left(\frac{n}{2^{\alpha}}\right)\right) &\mbox{for} & \beta\geqslant 2, \alpha\in\{0,1,2,3\},\vspace{1ex}\\
1 & \mbox{for} &\alpha \geqslant 4, \beta\geqslant 2 .\end{array}\right.$$
\end{enumerate}

\end{document}